\documentclass[11pt,reqno]{amsart}
\usepackage{fullpage}
\usepackage{mathrsfs,amssymb,graphicx,verbatim,amsmath,amsfonts}
\usepackage{paralist}
\usepackage[breaklinks,pdfstartview=FitH]{hyperref}
\usepackage{upgreek}
\usepackage{tikz}
\usepackage[mathscr]{euscript}

\hypersetup{ocgcolorlinks=true,allcolors=testc}
\hypersetup{
     colorlinks   = true,
     citecolor    = black
}
\hypersetup{linkcolor=black}
\hypersetup{urlcolor=black}

\usepackage{dutchcal}

\renewcommand{\le}{\leqslant}

\renewcommand{\leq}{\leqslant}
\renewcommand{\geq}{\geqslant}
\renewcommand{\setminus}{\smallsetminus}
\renewcommand{\gamma}{\upgamma}
\allowdisplaybreaks

\usepackage{esint}
\usepackage[autostyle]{csquotes}

\renewcommand{\pi}{\uppi}

\newcommand{\f}{\frac}

\newcommand{\e}{\varepsilon}
\newcommand{\R}{\mathbb R}

\newtheorem{theorem}{Theorem}
\newtheorem{lemma}[theorem]{Lemma}
\newtheorem{proposition}[theorem]{Proposition}

\newtheorem{corollary}[theorem]{Corollary}

\theoremstyle{remark}
\newtheorem{remark}[theorem]{Remark}

\newtheorem{question}[theorem]{Question}
\renewcommand{\tau}{\uptau}

\renewcommand{\xi}{\upxi}
\renewcommand{\rho}{\uprho}

\renewcommand{\subset}{\subseteq}

\newcommand{\N}{\mathbb N}

\newcommand{\eqdef}{\stackrel{\mathrm{def}}{=}}

\renewcommand{\theta}{\uptheta}
\renewcommand{\lambda}{\uplambda}

\renewcommand{\gamma}{\upgamma}
\renewcommand{\beta}{\upbeta}
\renewcommand{\alpha}{\upalpha}
\renewcommand{\kappa}{\upkappa}
\renewcommand{\psi}{\uppsi}
\renewcommand{\rho}{\uprho}
\renewcommand{\delta}{\updelta}
\renewcommand{\pi}{\uppi}
\renewcommand{\omega}{\upomega}
\renewcommand{\sigma}{\upsigma}

\renewcommand{\eta}{\upeta}
\renewcommand{\kappa}{\upkappa}
\renewcommand{\mu}{\upmu}
\renewcommand{\nu}{\upnu}
\renewcommand{\pi}{\uppi}
\renewcommand{\zeta}{\upzeta}

\newcommand{\mb}{\mathbb}
\newcommand*\diff{\mathop{}\!\mathrm{d}}

\newcommand{\ms}{\mathscr}
\newcommand{\msf}{\mathsf}

\newcommand{\mr}{\mathrm}

\begin{document}

\title{The dimensional Brunn--Minkowski inequality in Gauss space}

\author{Alexandros Eskenazis}
\address{(A.~E.) Institut de Math\'ematiques de Jussieu
\\ Sorbonne Universit\'e\\ 75252 Paris Cedex 05\\ France}
\email{alexandros.eskenazis@imj-prg.fr}

\author{Georgios Moschidis}
\address{(G.~M.) Department of Mathematics\\  University of California\\ Berkeley
\\ CA\\ 94720-3840\\ USA}
\email{gmoschidis@berkeley.edu}

\thanks{A.~E.~was supported by a postdoctoral fellowship of the Fondation Sciences Math\'ematiques de Paris. G.~M.~acknowledges support from the Miller Institute for Basic Research in Science, UC Berkeley.}

\maketitle

\vspace{-0.25in}

\begin{abstract}
Let $\gamma_n$ be the standard Gaussian measure on $\R^n$. We prove that for every symmetric convex sets $K,L$ in $\R^n$ and every $\lambda\in(0,1)$,
$$\gamma_n\big(\lambda K+(1-\lambda)L\big)^{\frac{1}{n}} \geq \lambda \gamma_n(K)^{\frac{1}{n}}+(1-\lambda)\gamma_n(L)^{\frac{1}{n}},$$
thus settling a problem raised by Gardner and Zvavitch (2010). This is the Gaussian analogue of the classical  Brunn--Minkowski inequality for the Lebesgue measure. We also show that, for a fixed $\lambda\in(0,1)$, equality is attained if and only if $K=L$.
\end{abstract}

\bigskip

{\footnotesize

}


\section{Introduction}

The classical Brunn--Minkowski inequality asserts that for every compact sets $A,B$ in $\R^n$ and every $\lambda\in(0,1)$,
\begin{equation} \label{eq:brunn-minkowski}
\big|\lambda A+(1-\lambda)B\big|^{\frac{1}{n}} \geq \lambda|A|^{\frac{1}{n}} + (1-\lambda)|B|^{\frac{1}{n}},
\end{equation}
where $|\cdot|$ denotes Lebesgue measure and the Minkowski convex combination of sets is given by
\begin{equation}
\lambda A+(1-\lambda)B = \big\{\lambda a+(1-\lambda)b: \ a,b\in A\big\}.
\end{equation}
In view of the importance of the Brunn--Minkowski inequality in convex geometric analysis (see the monographs \cite{Gar06, Sch14} and the surveys \cite{Gar02, Mau05, Bar06}), the last decades have seen a surge of activity around refinements and extensions of \eqref{eq:brunn-minkowski} in different contexts. We refer to \cite{KM17, KL18, HKL20} for an up to\mbox{ date account of some important achievements in this area.}

One prominent direction in modern Brunn--Minkowski theory is the study of inequalities relating the ``size'' of the Minkowski sum of subsets of $\R^n$ to the ``sizes'' of the individual summands, where ``size'' can be interpreted more loosely than in the sense of the usual Euclidean volume. In this paper, we will be interested in the case where ``size'' refers to the standard Gaussian measure $\gamma_n$ given by $\diff\gamma_n(x) = \tfrac{\exp(-|x|^2/2)}{(2\pi)^{n/2}}\diff x$; here $|x|$ denotes the Euclidean length of a vector $x\in\R^n$. An example of a profound geometric inequality for $\gamma_n$ is Ehrhard's inequality \cite{Ehr83}, which states that for every Borel measurable sets $A,B$ in $\R^n$ and every $\lambda\in(0,1)$,
\begin{equation} \label{eq:ehrhard}
\Phi^{-1}\big(\gamma_n\big(\lambda A+(1-\lambda)B\big)\big) \geq \lambda \Phi^{-1}\big(\gamma_n(A)\big)+(1-\lambda)\Phi^{-1}\big(\gamma_n(B)\big),
\end{equation}
where $\Phi^{-1}$ is the inverse of the Gaussian distribution function $\Phi(x)=\gamma_1((-\infty,x])$. Inequality \eqref{eq:ehrhard} is known to capture many delicate probabilistic and geometric properties of the Gaussian measure. For instance, it implies the Gaussian isoperimetric inequality, which asserts that half-spaces have minimal Gaussian surface area among all sets of fixed measure. Ehrhard's original proof of \eqref{eq:ehrhard} proceeded via Gaussian symmetrization and required both sets $A$ and $B$ to be convex, an assumption which was later removed by Borell in \cite{Bor03} (see also \cite{Lat96} for a partial result).

While Ehrhard's inequality \eqref{eq:ehrhard} captures the optimal dimension-free convexity of the Gaussian measure, its validity for general Borel subsets $A,B$ of $\R^n$ shows that it is oblivious to additional geometric properties of the underlying sets, such as convexity. In \cite{GZ10}, Gardner and Zvavitch undertook a  detailed investigation of Gaussian inequalities in (dual) Brunn--Minkowski theory, which they concluded by asking (see \cite[Question~7.1]{GZ10}) whether the dimensional Brunn--Minkowski inequality \eqref{eq:brunn-minkowski} holds with the Lebesgue measure $|\cdot|$ replaced by $\gamma_n$ with the assumption that the underlying sets are convex and contain the origin. A counterexample to this statement was produced by Nayar and Tkocz in \cite{NT13}, yet the possibility of such an inequality being true for all origin symmetric convex sets remained open. This problem is settled in the affirmative here.

\begin{theorem} \label{thm:g-z}
For every $n\in\N$, every symmetric convex sets $K,L$ in $\R^n$ and every $\lambda\in(0,1)$,
\begin{equation} \label{eq:g-z}
\gamma_n\big(\lambda K+(1-\lambda)L\big)^{\frac{1}{n}} \geq \lambda \gamma_n(K)^{\frac{1}{n}}+(1-\lambda)\gamma_n(L)^{\frac{1}{n}}.
\end{equation}
\end{theorem}

By taking $K$ and $L$ to be small coordinate boxes around the origin, it becomes clear that the exponent $\tfrac{1}{n}$ is optimal (i.e.~maximal) in inequality \eqref{eq:g-z}. As was already observed in \cite{GZ10}, the dimensional Gaussian Brunn--Minkowski inequality \eqref{eq:g-z} neither trivially follows nor implies Ehrhard's inequality \eqref{eq:ehrhard} for origin symmetric convex sets.


\subsection{Symmetry in Brunn--Minkowski theory} The dimensional Brunn--Minkowski inequality \eqref{eq:g-z} is a refinement of the log-concavity of the Gaussian measure, that is, the fact that for every Borel sets $A,B$ in $\R^n$ and every $\lambda\in(0,1)$,
\begin{equation} \label{eq:log-conc}
\gamma_n\big(\lambda A+(1-\lambda)B\big) \geq\gamma_n(A)^\lambda\gamma_n(B)^{1-\lambda},
\end{equation}
for the class of symmetric convex sets. Strengthenings of measure-theoretic inequalities such as \eqref{eq:log-conc} under convexity and symmetry assumptions repeatedly appear in Brunn--Minkowski theory and the geometry of the Gaussian measure. To illustrate this phenomenon, we recall the deep B-inequality of Cordero-Erausquin, Fradelizi and Maurey \cite{CFM04}, which states that for every symmetric convex set $K$ in $\R^n$, the map $\alpha(t)=\gamma_n(e^tK)$ is a log-concave function on $\R$. In this setting, the log-concavity of the function $\beta(s)=\gamma_n(sK)$ on $\R_+$ for every convex set $K$ in $\R^n$ is a straightforward consequence of \eqref{eq:log-conc}, yet the additional symmetry assumption is necessary for the log-concavity of $\alpha$ as can be seen by taking $K=[-1,\infty)$ on $\R$. Other important examples of inequalities in Gauss space (see also \cite{Lat02}) which crucially rely on the symmetry and convexity of the underlying sets are the S-inequality of Lata\l a and \mbox{Oleszkiewicz \cite{LO99} and Royen's correlation inequality \cite{Roy14}.}

Theorem \ref{thm:g-z} belongs in a large network of (largely conjectural) inequalities involving log-concave measures of various notions of convex combinations of symmetric convex sets in $\R^n$. At the top of the hierarchy of these inequalities lies the celebrated log-Brunn--Minkowski conjecture of B\"{o}r\"{o}czky, Lutwak, Yang and Zhang \cite{BLYZ12}, which asserts that the Euclidean volume of the geometric mean of two symmetric convex sets can be bounded below by the geometric mean of their volumes. Its validity would, for instance, imply the dimensional Brunn--Minkowski inequality \eqref{eq:g-z} and the B-inequality with $\gamma_n$ replaced by any symmetric log-concave measure $\mu$. Surveying in detail all recent developments in this area lies beyond the scope of this paper, so we refer to \cite{KM17, YZ19, CHLL18, Put19, BK20, HKL20, KL20} for recent results and further bibliographical information.


\subsection{Approaches towards the Gardner--Zvavitch problem} Since the formulation of the problem in \cite{GZ10}, there have been several partial results towards the general statement of Theorem \ref{thm:g-z}. Gardner and Zvavitch themselves proved the Gaussian Brunn--Minkowski inequality \eqref{eq:g-z} in the special cases that the sets $K,L$ are either coordinate boxes containing the origin or dilates of a fixed symmetric convex set. These results were later generalized to more general measures by Marsiglietti \cite{Mar16}. In \cite{CLM17}, Colesanti, Livshyts and Marsiglietti showed that \eqref{eq:g-z} holds when both symmetric bodies $K,L$ are small perturbations of the Euclidean ball. Moreover, Livshyts, Marsiglietti, Nayar and Zvavitch \cite{LMNZ17} have used a clever variant of the Pr\'ekopa--Leindler inequality \cite[Theorem~7.1.2]{Sch14} to show that inequality
\begin{equation} \label{eq:mu-b-mink}
\forall \ \lambda\in(0,1), \ \ \ \ \mu\big(\lambda A+(1-\lambda)B\big)^{\frac{1}{n}} \geq \lambda\mu(A)^{\frac{1}{n}}+(1-\lambda)\mu(B)^{\frac{1}{n}}
\end{equation}
holds true when $\mu$ is an unconditional product measure on $\R^n$ and $A,B$ are ideals in $\R^n$. Their result  was later extended to weakly unconditional sets by Ritor\'e and Yepes Nicol\'as \cite{RY18}. The planar case of inequality \eqref{eq:mu-b-mink} can be derived for any symmetric log-concave measure $\mu$ on $\R^2$ and all symmetric convex sets $K,L$ in $\R^2$ by combining \cite[Proposition~1]{LMNZ17} with \cite[Theorem~1.7]{BLYZ12} and a result of Saroglou \cite[Theorem~3.1]{Sar16} . Finally, the Gardner--Zvavitch problem was recently settled affirmatively for a class of symmetric convex sets with many hyperplane symmetries by B\"{o}r\"{o}czky and Kalantzopoulos \cite{BK20}.

\subsubsection{The local Gardner--Zvavitch problem, after Kolesnikov and Livshyts} The proofs of all the aforementioned results crucially require additional symmetries of the underlying sets which are not available in the general setting of Theorem \ref{thm:g-z}. In \cite{KL18}, Kolesnikov and Livshyts took a different route to attack the Gardner--Zvavitch problem, by studying how inequalities of the form
\begin{equation} \label{eq:dim-bm-mu}
\forall \ \lambda\in(0,1), \ \ \ \ \mu\big(\lambda K+(1-\lambda)L\big)^{\frac{\delta}{n}} \geq \lambda\mu(K)^{\frac{\delta}{n}}+(1-\lambda)\mu(L)^{\frac{\delta}{n}}
\end{equation}
behave infinitesimaly when the convex bodies $K$ and $L$ are small perturbations of each other and then proving a local-to-global principle. This is also the approach which we shall be taking. From now on, we will refer to twice continuously differentiable functions simply as smooth functions.

To illustrate this technique, we briefly return to the B-inequality of \cite{CFM04}, asserting that for every symmetric convex set $K$ in $\R^n$, the function $\alpha(t)=\gamma_n(e^tK)$ is log-concave on $\R$. It is straightforward to observe (see \cite[p.~413]{CFM04}) that the log-concavity of $\alpha$ on $\R$ for an arbitrary $K$ is equivalent to its infinitesimal log-concavity at $t=0$, that is, $\alpha''(0)\alpha(0)\leq\alpha'(0)^2$. An explicit\mbox{ calculation now shows that the latter inequality can be equivalently rewritten as}
\begin{equation} \label{eq:cfm}
\mathrm{Var}_{\gamma_K}(|x|^2) \leq \frac{1}{2} \int \big|\nabla |x|^2\big|^2 \,\diff\gamma_K(x),
\end{equation}
where $\gamma_K$ is the rescaled restriction of $\gamma_n$ on a symmetric convex set $K$ with nonempty interior, that is, the measure given by $\gamma_K(A)=\tfrac{\gamma_n(A\cap K)}{\gamma_n(K)}$ for Borel subsets $A$ of $\R^n$. The delicate aspect of inequality \eqref{eq:cfm} lies in the constant $\tfrac{1}{2}$ on the right hand side. Indeed, the same Poincar\'e inequality with constant 1 is valid in much greater generality by a classical result of Brascamp and Lieb \cite{BL76}, which implies that if $\mu$ is a measure of the form $\diff\mu(x)=e^{-V(x)}\diff x$ with a potential whose Hessian satisfies $\nabla^2V \geq \beta \msf{Id}$, then every smooth $f:\R^n\to\R$ satisfies
\begin{equation} \label{eq:brascamp-lieb}
\mathrm{Var}_{\mu}(f) \leq \frac{1}{\beta} \int |\nabla f|^2 \,\diff\mu.
\end{equation}
Equivalently, \eqref{eq:brascamp-lieb} says that the spectral gap of such a measure is at least $\sqrt{\beta}$. In order to prove \eqref{eq:cfm}, Cordero-Erausquin, Fradelizi and Maurey succeeded to realize this inequality as special case of a second eigenvalue problem for even functions and then crucially used the symmetry of both the measure $\gamma_K$ and the function $x\mapsto |x|^2$. Such an analytic use of the underlying symmetry of the problem also lies at the heart of the proof of Theorem \ref{thm:g-z} (see also Theorem \ref{thm:local-g-z} below).

Recall that the generator of the Ornstein--Uhlenbeck semigroup is the elliptic differential operator $\ms{L}$ whose action on a smooth function $u:\R^n\to\R$ is given by
\begin{equation}
\forall \ x\in\R^n, \ \ \ \ \ms{L}u(x) = \Delta u(x) - \sum_{i=1}^n x_i \partial_i u(x).
\end{equation}
We also denote by $\|A\|_{\mathrm{HS}}$ the Hilbert--Schmidt norm of a matrix $A$, i.e.~$\|A\|_{\mathrm{HS}}^2 = \sum_{i,j} a_{ij}^2$. In \cite{KL18}, Kolesnikov and Livshyts proved the following local-to-global principle (see also \cite{Col08, CLM17, KM18} for closely related infinitesimal versions of other Brunn--Minkowski-type inequalities).

\begin{proposition} [Kolesnikov, Livshyts] \label{prop:kl}
Fix $n\in\N$. Let $\delta\in[0,1]$ be such that for every symmetric convex set $K$ in $\R^n$, every smooth symmetric function $u:K\to\R$ with $\ms{L}u=1$ on $K$ satisfies
\begin{equation} \label{eq:local-g-z0}
\int \|\nabla^2u\|_{\mathrm{HS}}^2 + |\nabla u|^2 \,\diff\gamma_K \geq \frac{\delta}{n}.
\end{equation}
Then, for every symmetric convex sets $K,L$ in $\R^n$ and every $\lambda\in(0,1)$,
\begin{equation} \label{eq:gauss-b-m-again}
\gamma_n\big(\lambda K+(1-\lambda)L\big)^{\frac{\delta}{n}} \geq \lambda \gamma_n(K)^{\frac{\delta}{n}}+(1-\lambda)\gamma_n(L)^{\frac{\delta}{n}}.
\end{equation}
\end{proposition}

The main result of \cite{KL18} was a proof of \eqref{eq:local-g-z0} with $\delta=\tfrac{1}{2}$ for all convex sets $K$ containing the origin. Then, a local-to-global principle for such convex sets (similar to Proposition \ref{prop:kl}) implies the corresponding Gaussian Brunn--Minkowski inequality \eqref{eq:gauss-b-m-again} with $\delta=\tfrac{1}{2}$. The main technical result of the present paper is the following refinement of the inequality of Kolesnikov and Livshyts for origin symmetric convex sets and symmetric (i.e.~even) solutions of the equation $\ms{L}u=1$ on $K$.

\begin{theorem} \label{thm:local-g-z}
For every $n\in\N$ and every symmetric convex set $K$ in $\R^n$, every smooth symmetric function $u:\R^n\to\R$ with $\ms{L}u=1$ on $K$, satisfies
\begin{equation} \label{eq:local-g-z}
\int \|\nabla^2u\|_{\mathrm{HS}}^2 + |\nabla u|^2 \,\diff\gamma_K \geq \frac{1}{n}.
\end{equation}
\end{theorem}


\subsection{Beyond power-type concavity} While the exponent $\tfrac{1}{n}$ in inequality \eqref{eq:g-z} cannot be improved, Ehrhard's inequality \eqref{eq:ehrhard} suggests that it is worth investigating potential strengthenings of \eqref{eq:g-z}, where the power function $t\mapsto t^{\frac{1}{n}}$ is replaced by a more general function of the Gaussian measure of the sets. Notice that if $\zeta_n:[0,1]\to\R$ is such that for every Borel sets $A,B$ in $\R^n$, the  inequality
\begin{equation} \label{eq:zeta}
\zeta_n\big(\gamma_n\big(\lambda A+(1-\lambda)B\big)\big) \geq \lambda \zeta_n\big(\gamma_n(A)\big)+(1-\lambda)\zeta_n\big(\gamma_n(B)\big),
\end{equation}
holds true, then choosing $A$ and $B$ to be half-spaces of the form $A=\{x\in\R^n: \ x_1\leq a\}$ and $B=\{x\in\R^n: \ x_1\leq b\}$, we see that $\zeta_n\circ\Phi$ is concave; here $\Phi$ is the Gaussian distribution function $\Phi(x)=\gamma_1((-\infty,x])$. In this sense, the choice $\zeta_n=\Phi^{-1}$ in \eqref{eq:zeta}, encapsulated by Ehrhard's inequality \eqref{eq:ehrhard}, captures the optimal convexity of $\gamma_n$ over all Borel sets in $\R^n$. Bearing this in mind as a motivating example, we ask the following (purposefully vague) question.

\begin{question} \label{que}
Fix $n\in\N$. Is there an ``optimal'' increasing function $\xi_n:[0,1]\to\R$ such that for every origin symmetric convex sets $K,L$ in $\R^n$ and every $\lambda\in(0,1)$, the inequality
\begin{equation} \label{eq:quest}
\xi_n\big(\gamma_n(\lambda K+(1-\lambda)L\big)\big) \geq \lambda\xi_n(\gamma_n(K))+(1-\lambda)\xi_n(\gamma_n(L))
\end{equation}
is satisfied?
\end{question}

Ideally, such an inequality should be a joint refinement of Ehrhard's inequality \eqref{eq:ehrhard} and of the dimensional Brunn--Minkowski inequality \eqref{eq:g-z} which becomes an equality for some nontrivial pairs of symmetric convex sets $K, L$ in $\R^n$. In \cite[p.~5350]{GZ10}, Gardner and Zvavitch presented an argument of Barthe which implies that \eqref{eq:quest} is not satisfied for $\xi_n = \Psi_n^{-1}$, where $\Psi_n(r) = \gamma_n(rB_2^n)$ and $rB_2^n=\{x\in\R^n: \ |x|\leq r\}$ is the  Euclidean ball of radius $r$. We refer to Section \ref{sec:3} for some additional observations of this kind. While we have no conjecture as to what the optimal symmetric improvement \eqref{eq:quest} of Ehrhard's inequality \eqref{eq:ehrhard} might be, we obtain the\mbox{ following strengthening of \eqref{eq:g-z}.}

\begin{theorem} \label{thm:psi}
Fix $n\in\N$ and let $\sigma_n:[0,1]\to\R$ be a strictly increasing function satisfying
\begin{equation} \label{eq:def-psi}
\forall \ r\in(0,\infty), \ \ \ \ 1+\frac{\sigma_n''(\Psi_n(r)) \Psi_n(r)}{\sigma_n'(\Psi_n(r))} = \frac{2}{n} - \frac{c_n}{n^2\Psi_n(r)} r^n e^{-r^2/2},
\end{equation}
where $c_n^{-1} = 2^{\frac{n}{2}-1}\Gamma(n/2)$. Then, for every symmetric convex sets $K,L$ in $\R^n$ and every $\lambda\in(0,1)$,
\begin{equation} \label{eq:sigma}
\sigma_n\big(\gamma_n\big(\lambda K+(1-\lambda)L\big)\big) \geq \lambda \sigma_n(\gamma_n(K))+(1-\lambda)\sigma_n(\gamma_n(L)).
\end{equation}
\end{theorem}

An explicit calculation (see Remark \ref{rem:sigma}) reveals that the function $y\mapsto\sigma_n(y^n)$ is convex, hence \eqref{eq:sigma} is indeed a genuine improvement of the dimensional Brunn--Minkowski inequality \eqref{eq:g-z}. Moreover, the \emph{strict} convexity of $y\mapsto\sigma_n(y^n)$ readily implies the following corollary, settling the equality cases of the dimensional Brunn--Minkowski inequality \eqref{eq:g-z} for the Gaussian measure.

\begin{corollary} \label{cor}
Fix $n\in\N$ and $\lambda\in(0,1)$. If $K,L$ are two symmetric convex sets in $\R^n$ satisfying
\begin{equation}
\gamma_n\big(\lambda K+(1-\lambda)L\big)^{\frac{1}{n}} = \lambda \gamma_n(K)^{\frac{1}{n}}+(1-\lambda)\gamma_n(L)^{\frac{1}{n}},
\end{equation}
then $K=L$.
\end{corollary}

The proof of Theorem \ref{thm:local-g-z} will be presented in Section \ref{sec:2} and the proofs of Theorem \ref{thm:psi} and Corollary \ref{cor} in Section \ref{sec:2.5}. Some additional remarks are postponed to Section \ref{sec:3}.

\subsection*{Acknowledgements} We are grateful to Ramon van Handel for helpful discussions.


\section{Proof of Theorem \ref{thm:local-g-z}} \label{sec:2}

It follows readily from Proposition \ref{prop:kl} that Theorem \ref{thm:g-z} is a formal consequence of Theorem \ref{thm:local-g-z}, so in this section we shall only establish the latter. For an $n\times n$ matrix $A$, we will denote by $\widehat{A}$ the traceless part of $A$, that is $\widehat{A} = A-\tfrac{\mathrm{tr}(A)}{n}\msf{Id}$, where $\msf{Id}$ is the identity matrix. In particular, for a smooth function $u:K\to\R$, where $K$ is a symmetric convex set, we write
\begin{equation}
\widehat{\nabla}^2u \eqdef \nabla^2u - \frac{\Delta u}{n}\msf{Id}
\end{equation}
for the traceless part of its Hessian. Then, orthogonality implies the pointwise identity
\begin{equation} \label{eq:trace-decomp}
\|\nabla^2u\|_{\mr{HS}}^2 = \|\widehat{\nabla}^2u\|_{\mr{HS}}^2 + \frac{(\Delta u)^2}{n}.
\end{equation}
Consider $r:\R^n\to\R$ to be $r(x) =\tfrac{|x|^2}{2n}$, which satisfies $\widehat{\nabla}^2r \equiv 0$ and $\Delta r\equiv1$. Then, we have
\begin{equation} \label{eq:trace-decomp2}
 \|\widehat{\nabla}^2u\|_{\mr{HS}}^2 =  \big\|\widehat{\nabla}^2\big(u-r\big)\big\|_{\mr{HS}}^2 \stackrel{\eqref{eq:trace-decomp}}{=} \big\|\nabla^2\big(u-r\big)\big\|_{\mr{HS}}^2  - \frac{\big(\Delta(u-r)\big)^2}{n} =  \big\|\nabla^2\big(u-r\big)\big\|_{\mr{HS}}^2  - \frac{(\Delta u-1)^2}{n},
\end{equation}
so that combining \eqref{eq:trace-decomp} and \eqref{eq:trace-decomp2}, we get
\begin{equation}
\begin{split}
\|\nabla^2u\|_{\mathrm{HS}}^2  = \big\| \nabla^2\big(u-r\big)\big\|_{\mathrm{HS}}^2 +\frac{2}{n} \Delta u - \frac{1}{n}.
\end{split}
\end{equation}
Taking into account that $\ms{L}u(x) = \Delta u(x) -\sum_{i=1}^n x_i\partial_i u(x)=1$, we can then write
\begin{equation} \label{break-hessian}
\forall \ x\in K, \ \ \ \ \|\nabla^2u(x)\|_{\mathrm{HS}}^2 =  \big\| \nabla^2\big(u-r\big)(x)\big\|_{\mathrm{HS}}^2 +\frac{2}{n} \sum_{i=1}^n x_i\partial_i u(x) + \frac{1}{n}.
\end{equation}
For a fixed $i\in\{1,\ldots,n\}$ the partial derivative $\partial_i (u-r)$ is an odd function on $K$ and thus has expectation 0 with respect to $\gamma_K$. Therefore, the Brascamp--Lieb inequality \eqref{eq:brascamp-lieb} gives
\begin{equation} \label{eq:use-bl000}
\sum_{j=1}^n \int \big(\partial_j\partial_i (u-r)\big)^2\,\diff\gamma_K \geq \mathrm{Var}_{\gamma_K}\big(\partial_i(u-r)\big)  = \int \Big(\partial_iu(x)-\frac{x_i}{n}\Big)^2\,\diff\gamma_K(x).
\end{equation}
Summing \eqref{eq:use-bl000} over $i\in\{1,\ldots,n\}$, we get
\begin{equation} \label{just-used-bl}
\begin{split}
\int \|\nabla^2(u-r)\|_{\mr{HS}}^2\,\diff\gamma_K \geq \sum_{i=1}^n\int \Big(&\partial_iu(x)-\frac{x_i}{n}\Big)^2  \,\diff\gamma_K(x) 
\\ & = \int |\nabla u(x)|^2 - \frac{2}{n}\sum_{i=1}^n x_i \partial_i u(x) + \frac{|x|^2}{n^2} \,\diff\gamma_K(x)
\end{split}
\end{equation}
Combining \eqref{break-hessian} and \eqref{just-used-bl}, we finally deduce that
\begin{equation} \label{eq:last-th-3}
\int \|\nabla^2u\|_{\mr{HS}}^2 + |\nabla u|^2 \,\diff\gamma_K \geq \int 2|\nabla u(x)|^2  + \frac{|x|^2}{n^2} + \frac{1}{n} \,\diff\gamma_K(x) \geq \frac{1}{n}
\end{equation}
and the proof is complete. \hfill$\Box$

\begin{remark}
Notice that the only property of $K$ and $u$ that was used in the proof of Theorem \ref{thm:local-g-z} is the fact that $\int \nabla (u-r) \diff\gamma_K=0$, which allows us to use the Brascamp--Lieb inequality.
\end{remark}


\section{Proof of Theorem \ref{thm:psi}} \label{sec:2.5}

The proof of Theorem \ref{thm:psi}, in analogy with that of Theorem \ref{thm:g-z}, will proceed in two steps: first we will establish a suitable Poincar\'e-type inequality for solutions of the equation $\ms{L}u=1$ and then the geometric inequality \eqref{eq:sigma} will be a consequence of the following local-to-global principle.

\begin{proposition} \label{prop:psi-local-to-global}
Fix $n\in\N$. Let $\psi:[0,1]\to\R$ be an increasing function such that for every symmetric convex set $K$, every smooth symmetric function $u:K\to\R$ with $\ms{L}u=1$ on $K$ satisfies
\begin{equation}
\int \|\nabla^2u\|_{\mr{HS}}^2 + |\nabla u|^2 \,\diff\gamma_K \geq1+\frac{\psi''(\gamma_n(K))\gamma_n(K)}{\psi'(\gamma_n(K))}.
\end{equation}
Then, for every symmetric convex sets $K,L$ in $\R^n$ and every $\lambda\in(0,1)$,
\begin{equation}
\psi\big(\gamma_n(\lambda K+(1-\lambda)L\big)\big) \geq \lambda\psi(\gamma_n(K))+(1-\lambda)\psi(\gamma_n(L)).
\end{equation}
\end{proposition}

In order to prove Proposition \ref{prop:psi-local-to-global} one has to repeat the arguments of the proof of Proposition \ref{prop:kl} in \cite{KL18} (see also \cite{KM18}) mutatis mutandis, by replacing the power function $t\mapsto t^{\frac{\delta}{n}}$ by a general increasing function $\psi$. As the proof of the more general Proposition \ref{prop:psi-local-to-global} does not require any additional ideas, we leave the (trivial) necessary modifications to the interested reader.

\begin{lemma} \label{lem:balls}
Fix $n\in\N$ and let $K$ be a star-shaped set in $\R^n$. If $\rho\in[0,\infty]$ is such that $\gamma_n(K)=\gamma_n(\rho B_2^n)$, where $\rho B_2^n = \{x\in\R^n: \ |x|\leq \rho\}$ is the closed Euclidean ball of radius $\rho$, then
\begin{equation} \label{eq:lemballs}
\int_K |x|^2 \,\diff\gamma_n(x) \geq \int_{\rho B_2^n} |x|^2\,\diff\gamma_n(x).
\end{equation}
\end{lemma}

\begin{proof}
We will denote by $\mb{S}^{n-1}$ the unit sphere in $\R^n$ and by $\rho_K:\mb{S}^{n-1}\to[0,\infty]$ the radial function of $K$, i.e.~$\rho_K(\theta) = \sup\{r\geq0: \ r\theta\in K\}$. Let $A=\{\theta\in\mb{S}^{n-1}: \ \rho\leq\rho_K(\theta)\}$ \mbox{and $B=\mb{S}^{n-1}\setminus A$. Then,}
\begin{equation*}
\begin{split}
0 =(2\pi)^{n/2}\big( \gamma_n(K) - \gamma_n(\rho B_2^n) \big) & = \int_{\mb{S}^{n-1}} \int_0^{\rho_K(\theta)} r^{n-1} e^{-r^2/2}\,\diff r\diff\theta - \int_{\mb{S}^{n-1}}\int_0^\rho r^{n-1} e^{-r^2/2}\,\diff r\diff \theta
\\ & = \int_A \int_\rho^{\rho_K(\theta)} r^{n-1} e^{-r^2/2}\,\diff r\diff\theta - \int_B \int_{\rho_K(\theta)}^\rho r^{n-1} e^{-r^2/2}\,\diff r\diff\theta.
\end{split}
\end{equation*}
Therefore, using polar coordinates once again, we get
\begin{equation}
\begin{split}
(2\pi)^{n/2}\Big( \int_K |x|^2 & \,\diff\gamma_n(x)  - \int_{\rho B_2^n} |x|^2\,\diff\gamma_n(x)\Big) 
\\ & =  \int_A \int_\rho^{\rho_K(\theta)} r^{n+1} e^{-r^2/2}\,\diff r\diff\theta - \int_B \int_{\rho_K(\theta)}^\rho r^{n+1} e^{-r^2/2}\,\diff r\diff\theta
\\ & \geq \rho^2 \Big(\int_A \int_\rho^{\rho_K(\theta)} r^{n-1} e^{-r^2/2}\,\diff r\diff\theta - \int_B \int_{\rho_K(\theta)}^\rho r^{n-1} e^{-r^2/2}\,\diff r\diff\theta \Big) =0
\end{split}
\end{equation}
and the conclusion readily follows.
\end{proof}

We are now well equipped to complete the proof of Theorem \ref{thm:psi}.

\begin{proof} [Proof of Theorem \ref{thm:psi}]
Let $K$ be a symmetric convex set in $\R^n$ and $u:K\to\R$ a smooth symmetric function with $\ms{L}u=1$ on $K$. Then, by \eqref{eq:last-th-3} and Lemma \ref{lem:balls}, we have
\begin{equation} \label{eq:5-1}
\int \|\nabla^2u\|_{\mr{HS}}^2 + |\nabla u|^2 \,\diff\gamma_K \stackrel{\eqref{eq:last-th-3}}{\geq} \frac{1}{n^2} \int |x|^2\diff\gamma_K(x) + \frac{1}{n} \stackrel{\eqref{eq:lemballs}}{\geq} \frac{1}{n^2 \gamma_n(K)} \int_{\rho B_2^n} |x|^2 \diff\gamma_n(x) + \frac{1}{n},
\end{equation}
where $\gamma_n(K)=\gamma_n(\rho B_2^n)$ or, equivalently, $\rho = \Psi_n^{-1}(\gamma_n(K))$. Moreover, integration by parts gives
\begin{equation} \label{eq:5-2}
\begin{split}
\int_{\rho B_2^n} |x|^2\,\diff\gamma_n(x) = \frac{|\mb{S}^{n-1}|}{(2\pi)^{n/2}} \int_0^\rho r^{n+1} e^{-r^2/2}\,\diff r & =  \frac{|\mb{S}^{n-1}|}{(2\pi)^{n/2}} \Big(n\int_0^\rho r^{n-1}e^{-r^2/2} \,\diff r- \rho^n e^{-\rho^2/2}\Big)
\\ & = n\gamma_n(K) - c_n \rho^n e^{-\rho^2/2},
\end{split}
\end{equation}
where $c_n^{-1} = 2^{\frac{n}{2}-1}\Gamma(n/2)$. Combining \eqref{eq:5-1} and \eqref{eq:5-2}, we get
\begin{equation} \label{eq:get-psi}
\begin{split}
\int \|\nabla^2u\|_{\mr{HS}}^2 + |\nabla u|^2 \,\diff\gamma_K \geq \frac{2}{n} &- \frac{c_n}{n^2\gamma_n(K)}\rho^n e^{-\rho^2/2} 
\\ & =  \frac{2}{n} - \frac{c_n}{n^2\gamma_n(K)} \Psi_n^{-1}(\gamma_n(K))^n e^{-\Psi_n^{-1}(\gamma_n(K))^2/2}
\end{split}
\end{equation}
and the conclusion follows readily by the definition \eqref{eq:def-psi} of $\sigma_n$ and Proposition \ref{prop:psi-local-to-global}.
\end{proof}

\begin{remark} \label{rem:sigma}
To see that \eqref{eq:sigma} is a strict strengthening of \eqref{eq:g-z}, it suffices to observe that the function $\tau_n(x) = \sigma_n^{-1}(x)^{\frac{1}{n}}$ is increasing and concave, since then we can write
\begin{equation} \label{tau-is-concave}
\begin{split}
\gamma_n\big(\lambda K+(1-&\lambda)L\big)^{\frac{1}{n}}  = \tau_n \circ\sigma_n\big(\gamma_n(\lambda K+(1-\lambda)L)\big)\\& \stackrel{\eqref{eq:sigma}}{\geq} \tau_n\big(\lambda\sigma_n(\gamma_n(K)) + (1-\lambda)\sigma_n(\gamma_n(L)) \big)
\\ & \geq \lambda \tau_n\circ\sigma_n \big(\gamma_n(K)\big) + (1-\lambda) \tau_n\circ\sigma_n\big(\gamma_n(L)\big) = \lambda\gamma_n(K)^\frac{1}{n} + (1-\lambda)\gamma_n(L)^\frac{1}{n}.
\end{split}
\end{equation}
In order to prove that $\tau_n$ is concave, we will instead show that $\tau_n^{-1}(y) = \sigma_n(y^n)$ is convex. Indeed, the second derivative condition $(\tau_n^{-1})''(y)\geq0$ can be equivalently rewritten as
\begin{equation} \label{eq:second-deriv}
\forall \ y\in(0,\infty), \ \ \ \ \frac{\sigma_n''(y^n) y^n}{\sigma_n'(y^n)} \geq - \frac{n-1}{n},
\end{equation}
which readily follows from \eqref{eq:def-psi} and \eqref{eq:5-2}.
\end{remark}

\begin{proof} [Proof of Corollary \ref{cor}]
Using the notation of the previous remark, we see from \eqref{eq:5-1} and \eqref{eq:second-deriv} that $\tau_n$ is in fact a strictly concave and strictly increasing function. Suppose now that for some symmetric convex sets $K,L$ in $\R^n$, we have
\begin{equation} \label{eq:equal-condi}
\gamma_n\big(\lambda K+(1-\lambda)L\big)^{\frac{1}{n}} = \lambda \gamma_n(K)^{\frac{1}{n}}+(1-\lambda)\gamma_n(L)^{\frac{1}{n}}.
\end{equation}
Then, since the second inequality in \eqref{tau-is-concave} becomes equality, we deduce that $\sigma_n(\gamma_n(K)) = \sigma_n(\gamma_n(L))$, which implies that $\gamma_n(K)=\gamma_n(L)$. Combining this with the equality condition \eqref{eq:equal-condi}, we have
\begin{equation}
\gamma_n(K) = \gamma_n\big(\lambda K+(1-\lambda)L\big) = \gamma_n(L),
\end{equation}
which can only hold when $K=L$, e.g. by \cite{Dub77} or \cite{Ehr86}.
\end{proof} 


\section{Further remarks} \label{sec:3}

We conclude with a few additional remarks on Brunn--Minkowski-type inequalities for measures.

\smallskip

\noindent {\bf 1.} In \cite{KL18}, Kolesnikov and Livshyts showed that when $K$ is a convex set containing the origin and $u:K\to\R$ is a smooth function satisfying $\ms{L}u=1$ on $K$, then 
\begin{equation}
\int \|\nabla^2u\|_{\mr{HS}}^2 + |\nabla u|^2 \,\diff\gamma_K \geq\frac{1}{2n}
\end{equation}
in the following way. By omitting the traceless part of $\nabla^2u$, \eqref{eq:trace-decomp} gives 
\begin{equation*}
\int \|\nabla^2u\|_{\mr{HS}}^2 + |\nabla u|^2 \,\diff\gamma_K \stackrel{\eqref{eq:trace-decomp}}{\geq} \int \frac{(\Delta u)^2}{n} +|\nabla u|^2\,\diff\gamma_K = \int \frac{1}{n}\Big(1+\sum_{i=1}^n x_i\partial_iu(x)\Big)^2 + |\nabla u(x)|^2\,\diff\gamma_K(x),
\end{equation*}
where in the equality we used that $\ms{L}u=1$. Pointwise minimizing the right hand side over all $u_1,\ldots,u_n\in\R$ (see \cite[Lemma~2.4]{KL18}), they get
\begin{equation} \label{eq:galyna-sasha}
\ms{G}(u) \eqdef \int \frac{1}{n}\Big(1+\sum_{i=1}^n x_i\partial_iu(x)\Big)^2 + |\nabla u(x)|^2\diff\gamma_K(x) \geq \int \frac{1}{|x|^2+n}\,\diff\gamma_K(x)
\end{equation}
and this quantity is always greater than $\tfrac{1}{2n}$ by a simple geometric inequality \cite[Lemma~5.3]{KL18}. It can easily be seen that when $K=\R^n$ this quantity is in fact $\tfrac{1}{2n}+ o\big(\tfrac{1}{n}\big)$. One indication that this argument can be improved is the fact that inequality \eqref{eq:galyna-sasha} is saturated for the function $u_0(x) = -\tfrac{1}{2}\log(|x|^2+n)$ that satisfies
\begin{equation}
\forall \ x\in\R^n, \ \ \ \ \ms{L}u_0(x) = \frac{|x|^2-n}{|x|^2+n}+\frac{2|x|^2}{(|x|^2+n)^2}
\end{equation}
and this is very far from the imposed constraint $\ms{L}u=1$ (e.g.~when $|x|\sim\sqrt{n}$). 

In the proof of Theorem \ref{thm:local-g-z} above, the traceless part $\widehat{\nabla}^2u$ of the Hessian of $u$ is controlled via the Brascamp--Lieb inequality \eqref{eq:brascamp-lieb}. The following proposition shows that by crudely omitting it, one cannot substantially improve the result of \cite{KL18}, even for very simple symmetric convex sets. In what follows, we will adopt the following shorthand notation: Given any $x=(x_1, x_2, \ldots, x_n)\in \mathbb{R}^n$, we will denote by $x'$ the vector $(x_2, x_3, \ldots, x_n)\in \mathbb{R}^{n-1}$, so that $x = (x_1, x')$. For any $\e\in(0,1)$, let $\ms{S}_{\e}=\{|x_1|<\e\}\cap\{|x'|^2<2n\}\subset \mathbb{R}^n$ be a truncated slab of width $\e$. 

\begin{proposition}
There exists an absolute constant $C\in(0,\infty)$, such that, for every $n\geq2$ and every small enough $\e\in(0,1)$, the functional $\ms{G}$ associated to the slab $\ms{S}_{\e}$ in $\R^n$ satisfies
\begin{equation}
\inf \big\{\ms{G}(u): \ u\in C^2(\ms{S}_{\e}) \mbox{ with }  \ms{L}u=1\big\} \le \frac{1}{2n}+C\e^2+o\Big(\frac{1}{n}\Big). \label{Upper Bound G}
\end{equation}
\end{proposition}

\begin{proof}
Let us define the function $u:\ms{S}_{\e} \rightarrow \mathbb{R}$ by the relation $u \eqdef u_0 + v$, where we denote by $u_0(x) = -\tfrac{1}{2} \log (|x|^2+n)$ the pointwise minimizer of the functional $\ms{G}$ 
and the function $v$ is the unique solution of the boundary value problem
\begin{equation}
\begin{cases}
\ms{L}{v} = R,\label{Johny BV problem}\\
v|_{\partial{\ms{S}_{\e}}} = 0,
\end{cases}
\end{equation}
where 
\begin{equation}
R(x) \eqdef \f{2n}{|x|^2+n}-\f{2|x|^2}{(|x|^2+n)^2}
\end{equation}
(the existence and uniqueness of $v$ follows from \cite[Theorem~6.13]{GilbargTrudinger}). Note that the function $u$ satisfies $\ms{L}u=1$ and that, by straightforward algebra, $\ms{G}(u)$ can be expressed as 
\begin{equation}
\ms{G}(u) = \f1n \int H\Big(\f{|x|^2}{n} \Big)\,\diff\gamma_{\ms{S}_{\e}}(x) + \int  \f1n \Big(\sum_{i=1}^n x_i \partial_i v(x)\Big)^2 +|\nabla v(x)|^2 \,\diff\gamma_{\ms{S}_{\e}}(x), \label{Functional G}
\end{equation}
where $H(t) = \f{1}{t+1}$ for $t\geq0$.

Using the fact that $|x'|^2\le|x|^2\le|x'|^2+\e^2$ and $|x|^2<2n$ on $\ms{S}_{\e}$, we can bound 
\begin{equation}
\int H\Big(\f{|x|^2}{n} \Big) \,\diff\gamma_{\ms{S}_{\e}}(x)\le \int H\Big(\f{|x'|^2}{n} \Big) \,\diff\gamma_{\ms{S}_{\e}}(x)+ \f{\e^2}{n}, \label{First Bound H}
\end{equation}
since $H$ is 1-Lipschitz on the interval $[0,2]$. Furthermore, we can estimate
\begin{equation}
\begin{split}
\int & H\Big(\f{|x'|^2}{n} \Big) \,\diff\gamma_{\ms{S}_{\e}}(x) = \f{1}{\gamma_n(\ms{S}_{\e})}\int_{-\e}^{\e} \int_{\{|x'|^2<2n\}} H\Big(\f{|x'|^2}{n} \Big) \,\diff\gamma_{n-1}(x') \diff\gamma_1(x_1)\\
 &=\fint_{\big\{\f{|x'|^2}{n}<2\big\}} H\Big( \f{|x'|^2}{n}\Big) \diff\gamma_{n-1}(x')
 = H(1)+ \fint_{\big\{\f{|x'|^2}{n}<2\big\}} \Big( H\Big( \f{|x'|^2}{n}\Big) - H(1)\Big) \diff\gamma_{n-1}(x')\\
 & \le H(1)+ \fint_{\big\{\f{|x'|^2}{n}<2\big\}} \left|\f{|x'|^2}{n}-1\right| \diff\gamma_{n-1}(x') \le \f{1}{2} + o(1), \label{x^2=n}
\end{split}
\end{equation}
as $n\to\infty$. Thus, from (\ref{First Bound H}) and (\ref{x^2=n}), we obtain the following bound for the first summand of (\ref{Functional G}),
\begin{equation}
\f1n \int H\Big(\f{|x|^2}{n} \Big)\,\diff\gamma_{\ms{S}_{\e}}(x) \le \f{1}{2n} +\f{\e^2}{n} + o\Big(\f1n\Big), \label{First Summand}
\end{equation}
as $n\to\infty$. In order to complete the proof of (\ref{Upper Bound G}), it suffices to show that the second summand of (\ref{Functional G}) satisfies
\begin{equation}
\int \f1n \Big(\sum_{i=1}^n x_i \partial_i v(x)\Big)^2 +|\nabla v(x)|^2\,\diff\gamma_{\ms{S}_{\e}}(x) \le C \e^2 \label{Smallness Poincare Final}
\end{equation} 
for some absolute constant $C>0$, provided that $\e \ll 1$. In view of the bound
\begin{equation}
\int \f1n \Big(\sum_{i=1}^n x_i \partial_i v(x)\Big)^2 +|\nabla v(x)|^2\,\diff\gamma_{\ms{S}_{\e}}(x) \le \int  \Big(\f{|x|^2}{n}+1\Big)|\nabla v(x)|^2\,\diff\gamma_{\ms{S}_{\e}}(x) \le 4  \int |\nabla v|^2 \,\diff\gamma_{\ms{S}_{\e}},
\end{equation}
which follows from the upper bound $|x|^2<2n+\e^2$ on $\ms{S}_{\e}$, it suffices to establish that
\begin{equation}
\int |\nabla v|^2 \, \diff\gamma_{\ms{S}_{\e}} \le C \e^2.
\end{equation} 
Therefore, the following lemma completes the proof of  (\ref{Upper Bound G}).
\end{proof}

\begin{lemma}
For any $\e>0$, the function $v:\ms{S_{\e}}\rightarrow \mathbb{R}$ defined by (\ref{Johny BV problem}) satisfies 
\begin{equation}
\int |\nabla v|^2 \, \diff\gamma_{\ms{S}_{\e}} \le 36 e^{\e^2/2} \e^2 \label{Poincare Smallness}
\end{equation}
\end{lemma}

\begin{proof}

Let us multiply equation  (\ref{Johny BV problem}) with $v$ and integrate over $\ms{S}_{\e}$, obtaining the relation
\begin{equation}
\int v \cdot \ms{L}v \, \diff\gamma_{\ms{S}_{\e}} = \int v \cdot R \, \diff\gamma_{\ms{S}_{\e}}. \label{Before Multiplier}
\end{equation}
Integrating by parts in the left hand side of (\ref{Before Multiplier}) then yields
\begin{equation}
\begin{split}
\int |\nabla v|^2 \, \diff\gamma_{\ms{S}_{\e}} &= -\int v \cdot R \, \diff\gamma_{\ms{S}_{\e}}  \le \Big(\int v^2 \, \diff\gamma_{\ms{S}_{\e}}\Big)^{\f12}\Big(\int R^2 \, \diff\gamma_{\ms{S}_{\e}}\Big)^{\f12}. \label{Multiplier}
\end{split}
\end{equation}
Since $v$ vanishes on $\{x_1 = -\e\}\cap\{|x'|^2 < 2n\} \subset \partial \ms{S}_{\e}$, for any $x=(x_1, x')\in \ms{S}_{\e}$, we have
\begin{equation}
\begin{split}
|v(x_1,x')| = \Big|\int_{-\e}^{x_1} \partial_1 v(s,x') \, \diff s\Big| & \le \int_{-\e}^{\e} \big|\partial_1 v(s,x')\big| \, \diff s\\ \label{Fundamental theorem of calculus}
 & \le  e^{\e^2/4}\int_{-\e}^{\e} e^{-s^2/4}\big|\partial_1 v(s,x')\big| \, \diff s.
\end{split}
\end{equation}
Therefore, using (\ref{Fundamental theorem of calculus}), we can estimate 
\begin{equation}
\begin{split}
\int v^2  \,\diff\gamma_{\ms{S}_{\e}} & = \f{1}{\gamma_n({\ms{S}_{\e}})} \int_{\ms{S}_{\e}} v^2(x_1,x') \,\diff\gamma_n(x_1,x') \\
 &\stackrel{(\ref{Fundamental theorem of calculus})}{\le}  \f{1}{\gamma_n({\ms{S}_{\e}})}\int_{\ms{S}_{\e}} e^{\e^2/2}\Big( \int_{-\e}^{\e} e^{-s^2/4} \big|\partial_1 v(s,x')\big| \,  \diff s \Big)^2 \,\diff\gamma_n(x_1,x') \\
 & \le \f{e^{\e^2/2}}{\gamma_n({\ms{S}_{\e}})}\int_{\ms{S}_{\e}} 2\e \Big( \int_{-\e}^{\e} e^{-s^2/2}(\partial_1 v(s,x'))^2 \,  \diff s \Big) \,\diff\gamma_n(x_1,x') \\
 & = \f{2e^{\e^2/2}\e}{\gamma_n({\ms{S}_{\e}})}\int_{\{|x'|^2<2n\}}\int_{-\e}^{\e} \Big( \int_{-\e}^{\e} e^{-s^2/2}(\partial_1 v(s,x'))^2 \,  \diff s \Big) \,\diff\gamma_1(x_1)\, \diff\gamma_{n-1}(x')\\
& = \f{2e^{\e^2/2}\e}{\gamma_n({\ms{S}_{\e}})}\int_{\{|x'|^2<2n\}}\Big( \int_{-\e}^{\e} \diff\gamma_1(x_1)\Big) \Big( \int_{-\e}^{\e} e^{-s^2/2}(\partial_1 v(s,x'))^2 \,  \diff s \Big) \,\diff\gamma_{n-1}(x') \\
 & \le \f{2e^{\e^2/2}\e}{\gamma_n({\ms{S}_{\e}})} \int_{\{|x'|^2<2n\}} \f{2\e}{\sqrt{2\pi}}\Big( \int_{-\e}^{\e} e^{-s^2/2}(\partial_1 v(s,x'))^2 \,  \diff s \Big) \,\diff\gamma_{n-1}(x')\\
& = \f{4e^{\e^2/2}\e^2}{\gamma_n({\ms{S}_{\e}})} \int_{\{|x'|^2<2n\}} \int_{-\e}^{\e} (\partial_1 v(s,x'))^2 \, \diff\gamma_1 (s) \,\diff\gamma_{n-1}(x') \le 4e^{\e^2/2}\e^2 \int |\nabla v|^2 \,\diff\gamma_{\ms{S}_{\e}}. \label{Improved Poincare}
\end{split}
\end{equation}
Finally, using the trivial pointwise bound 
\begin{equation}
|R(x)|\le  \f{2n}{|x|^2+n}  +  \f{|x|^2}{(|x|^2+n)^2} < 2 +\f{1}{n} \le 3
\end{equation}
along with (\ref{Improved Poincare}), we obtain from (\ref{Multiplier}) that
\begin{equation}
\int |\nabla v|^2 \, \diff\gamma_{\ms{S}_{\e}} \le 6e^{\e^2/4}\e \Big(\int |\nabla v|^2 \, \diff\gamma_{\ms{S}_{\e}}\Big)^{\f12} \label{Final Bound Slab}
\end{equation}
and the conclusion readily follows.
\end{proof}

\smallskip

\noindent {\bf 2.} The bound of Theorem \ref{thm:local-g-z} can be sharpened when $K$ is $\rho B_2^n$, the Euclidean ball of radius $\rho\in(0,\infty)$. Let $u$ be a smooth symmetric function with $\ms{L}u=1$ on $\rho B_2^n$ and $u(0)=0$. Then, for every $r\in[0,\rho)$ and $\theta\in\mb{S}^{n-1}$, we can write $u(r\theta) = u_0(r) + v(r\theta)$, where
\begin{equation}
\forall \ r\in[0,\rho), \ \ \ \ u_0(r) = \frac{1}{|\mb{S}^{n-1}|} \int_{\mb{S}^{n-1}} u(r\phi)\,\diff\phi \ \ \ \mbox{and} \ \ \ \int_{\mb{S}^{n-1}} v(r\phi)\,\diff\phi=0.
\end{equation}
For every $r\in(0,\rho)$, we have
\begin{equation} \label{eq:u_0-eq}
u_0''(r) + \Big(\frac{n-1}{r}-r\Big) u_0'(r) = \ms{L}u_o(r) = \frac{1}{|\mb{S}^{n-1}|} \int_{\mb{S}^{n-1}} \ms{L}u(r\phi)\,\diff\phi = 1,
\end{equation}
hence also $\ms{L}v = \ms{L}u-\ms{L}u_0=0$. Moreover, notice that $u_0(0)=0$ and
\begin{equation}
u_0'(0)=\lim_{r\to0} \frac{1}{|\mb{S}^{n-1}|} \int_{\mb{S}^{n-1}}  \sum_{i=1}^n \frac{\phi_ix_i}{r} \partial_iu(r\phi)\,\diff\phi = \frac{1}{|\mb{S}^{n-1}|} \int_{\mb{S}^{n-1}}  \sum_{i=1}^n \phi_i \lim_{r\to0} \frac{x_i}{r}\partial_iu(r\phi)\,\diff\phi=0,
\end{equation}
since $u$ is even and thus $\partial_iu(0)=0$. Additionally, It can be shown that we can write
\begin{equation}
\int \|\nabla^2u\|_{\mr{HS}}^2 + |\nabla u|^2 \,\diff\gamma_{\rho B_2^n} = \int \|\nabla^2u_0\|_{\mr{HS}}^2 + |\nabla u_0|^2 \,\diff\gamma_{\rho B_2^n} + \int \|\nabla^2v\|_{\mr{HS}}^2 + |\nabla v|^2 \,\diff\gamma_{\rho B_2^n}.
\end{equation}
Combining the above observations, we deduce that
\begin{equation*}
\min\left\{ \int \|\nabla^2u\|_{\mr{HS}}^2 + |\nabla u|^2 \,\diff\gamma_{\rho B_2^n}: \ u\in C^2(\rho B_2^n) \mbox{ with }\ms{L}u=1 \right\} =  \int \|\nabla^2u_0\|_{\mr{HS}}^2 + |\nabla u_0|^2 \,\diff\gamma_{\rho B_2^n},
\end{equation*}
where $u_0$ is the unique solution of \eqref{eq:u_0-eq} satisfying $u_0(0)=u_0'(0)=0$. An explicit calculation now shows that this function $u_0$ satisfies
\begin{equation}
 \int \|\nabla^2u_0\|_{\mr{HS}}^2 + |\nabla u_0|^2 \,\diff\gamma_{\rho B_2^n} = 1-\left(\frac{n-1}{\rho}-\rho\right) \frac{\int_0^\rho r^{n-1}e^{-r^2/2}\diff r}{\rho^{n-1}e^{-\rho^2/2}} > \frac{1}{n}.
\end{equation}

\smallskip

\noindent {\bf 3.} In view of the simplicity of the proof of Theorem \ref{thm:local-g-z}, it is natural to wonder whether a similar argument can be employed to prove dimensional Brunn--Minkowski inequalities for measures other than the Gaussian. If $\mu$ is a symmetric log-concave measure on $\R^n$ with $\diff\mu(x) =e^{-V(x)}\diff x$, the action of the corresponding elliptic operator $\ms{L}_\mu$ on a smooth function $u:\R^n\to\R$ is given by
\begin{equation}
\forall x\in\R^n, \ \ \ \ \ms{L}_\mu u(x) = \Delta u(x) - \langle \nabla V(x),\nabla u(x)\rangle.
\end{equation}
Proposition \ref{prop:kl} is a special case of the following more general statement (see \cite{KL18}). Let $\mu$ be a symmetric log-concave measure on $\R^n$ with $\diff\mu(x) =e^{-V(x)}\diff x$ and suppose that for every symmetric convex set $K$ in $\R^n$, every smooth symmetric function $u:K\to\R$ with $\ms{L}_\mu u=1$ on $K$ satisfies
\begin{equation} \label{eq:lb for mu}
\int \|\nabla^2 u\|_{\mathrm{HS}}^2 + \langle \nabla^2 V\nabla u,\nabla u\rangle\,\diff\mu_K \geq \frac{\delta}{n},
\end{equation}
for some $\delta\in[0,1]$; here $\mu_K$ is the rescaled restriction of $\mu$ on $K$. Then, for every symmetric convex sets $K,L$ in $\R^n$ and $\lambda\in(0,1)$, we have
\begin{equation}
\mu\big(\lambda K+(1-\lambda)L\big)^{\frac{\delta}{n}} \geq \lambda\mu(K)^{\frac{\delta}{n}}+(1-\lambda)\mu(L)^{\frac{\delta}{n}}.
\end{equation}
Suppose now that the Hessian of the potential of $\mu$ satisfies $\nabla^2V \geq\msf{Id}$. Then, arguing exactly as in the proof of Theorem \ref{thm:local-g-z} via \eqref{break-hessian} and the Brascamp-Lieb inequality \eqref{eq:brascamp-lieb} applied to $\nabla(u-r)$, we see that every smooth function $u:K\to\R$ with $\ms{L}_\mu u=1$ on $K$ satisfies
\begin{equation}
\begin{split}
\int & \|\nabla^2 u\|_{\mathrm{HS}}^2 + \langle \nabla^2 V\nabla u,\nabla u\rangle\diff\mu
\\ & \geq \int |\nabla u(x)|^2+\langle \nabla^2 V(x)\nabla u(x),\nabla u(x)\rangle + \frac{2}{n} \langle \nabla V(x)-x,\nabla u(x)\rangle + \frac{|x|^2}{n^2}+\frac{1}{n}\,\diff\mu_K(x).
\end{split}
\end{equation}
However, this estimate is not sufficient to deduce inequality \eqref{eq:lb for mu} with $\delta=1$ even for $n=2$. To see this, fix $N>\!\!\!>1$ and consider the log-concave probability measure $\diff\mu(x) = ce^{-Nx_1^2-x_2^2}\diff x$, where $c$ is a normalizing constant, and the planar rectangle $\ms{R} =[-a,a]\times[-b,b]$ with $a=N^{-1/2}$ and $b=e^{-N}$. Let $u:\ms{R}\to\R$ be the unique solution of the boundary value problem
\begin{equation}
\begin{cases}
\ms{L}_\mu{u} = 1, & \mbox{on } \ms{R}\label{Johny BV problem 2}\\
\partial_1u(a,x_2) = - \partial_1u(-a,x_2) = -N^{-4/3}, & \mbox{for } x_2\in(-b,b)\\
\partial_2u(x_1,b) = -\partial_2u(x_1,-b) = \eta_N, & \mbox{for } x_1\in(-a,a)
\end{cases},
\end{equation}
where $\eta_N$ satisfies the necessary compatibility condition
\begin{equation}
\frac{\exp(-Na^2/2)}{\int_{-a}^a \exp(-Nx_1^2/2)\,\diff x_1}N^{-4/3} = \frac{\exp(-b^2/2)}{\int_{-b}^b \exp(-x_2^2/2)\,\diff x_2} \eta_N - \frac{1}{2}.
\end{equation}
Then, explicit (though tedious) computations show that
\begin{equation}
\int \langle \nabla V(x)-x,\nabla u(x)\rangle\,\diff\mu_\ms{R}(x) \lesssim -\frac{1}{N^{4/3}},
\end{equation}
which in turn implies
\begin{equation}
 \int |\nabla u(x)|^2+\langle \nabla^2 V(x)\nabla u(x),\nabla u(x)\rangle + \langle \nabla V(x)-x,\nabla u(x)\rangle + \frac{|x|^2}{4}\,\diff\mu_\ms{R}(x) \lesssim -\frac{1}{N^{4/3}},
\end{equation}
as all the remaining positive terms are of order at most $N^{-3/2}$. 

This reveals that the cancelations in the proof of Theorem \ref{thm:local-g-z} are specific to the Gaussian measure and in the case of a general symmetric log-concave measure $\mu$ with a lower bound on the Hessian of its potential, one would need to find a more delicate way to control the norm of $\widehat{\nabla}^2u$ than the vanilla application \eqref{eq:trace-decomp2}--\eqref{just-used-bl} of the Brascamp--Lieb inequality \eqref{eq:brascamp-lieb} on $\nabla(u-r)$.

\smallskip

\noindent {\bf 4.} If $\lambda\in(0,1)$, the $\lambda$-geometric mean of two symmetric convex sets $K,L$ in $\R^n$ is the set
\begin{equation}
K^\lambda L^{1-\lambda} \eqdef \bigcap_{\theta\in\mb{S}^{n-1}} \big\{x\in\R^n: \ \langle x,\theta\rangle \leq h_K(\theta)^\lambda h_L(\theta)^{1-\lambda}\big\},
\end{equation}
where the support function $h_M:\mb{S}^{n-1}\to\R_+$ of a symmetric convex set $M$ is $h_M(\theta) = \sup_{y\in M} \langle y,\theta\rangle$. The log-Brunn--Minkowski conjecture of \cite{BLYZ12}, asserts that for any such sets and $\lambda\in(0,1)$,
\begin{equation} \label{eq:log-bm}
\big| K^\lambda L^{1-\lambda}\big| \geq |K|^\lambda |L|^{1-\lambda}.
\end{equation}
In \cite{Sar16}, Saroglou proved that the log-Brunn--Minkowski conjecture formally implies the validity of \eqref{eq:log-bm}, with the Lebesgue measure $|\cdot|$ replaced by any symmetric log-concave measure $\mu$, that is
\begin{equation} \label{eq:log-bm2}
\mu\big(K^\lambda L^{1-\lambda}\big) \geq \mu(K)^\lambda \mu(L)^{1-\lambda}.
\end{equation}
Then, in \cite[Proposition~1]{LMNZ17}, the authors showed that \eqref{eq:log-bm2} implies the dimensional Brunn--Minkowski inequality \eqref{eq:dim-bm-mu} with $\delta=1$ for all symmetric convex sets $K,L$ in $\R^n$. We note here that in the special case where $\mu=\gamma_n$, one gets a slightly stronger result in the spirit of Question \ref{que}. Recall that $\Psi_n:[0,\infty)\to[0,1)$ is given by $\Psi_n(r) = \gamma_n(rB_2^n)$, where $rB_2^n=\{x\in\R^n: \ |x|\leq r\}$. We will need the following well-known lemma, which we could not locate in the literature.

\begin{lemma} \label{lem:rev-S}
Fix $n\in\N$ and let $A$ be a Borel measurable set in $\R^n$. If $\rho\in[0,\infty]$ is such that $\gamma_n(A)=\gamma_n(\rho B_2^n)$, then for every $t\in[0,1]$, we have $\gamma_n(tA)\geq\gamma_n(t\rho B_2^n)$.
\end{lemma}

\begin{proof}
Let $B=\rho B_2^n$. The assumption implies that $\gamma_n(A\setminus B)=\gamma_n(B\setminus A)$. Therefore,
\begin{equation}
\begin{split}
(2\pi)^{n/2}\big(\gamma_n(tA)-\gamma_n(tB)\big) & = \int_{t(A\setminus B)} e^{-|x|^2/2}\,\diff x - \int_{t(B\setminus A)} e^{-|x|^2/2}\,\diff x
\\ & = t^n\Big(\int_{A\setminus B} e^{-t^2 |y|^2/2} \,\diff y - \int_{B\setminus A} e^{-t^2 |y|^2/2}\,\diff y\Big)
\\ & \geq t^n e^{(1-t^2)\rho^2/2} \Big(\int_{A\setminus B} e^{- |y|^2/2} \,\diff y - \int_{B\setminus A} e^{- |y|^2/2}\,\diff y\Big)=0
\end{split}
\end{equation}
and the proof is complete.
\end{proof}

\begin{proposition}
Fix $n\in\N$. If \eqref{eq:log-bm2} holds true for $\mu=\gamma_n$, then for every symmetric convex sets $K,L$ in $\R^n$ and every $\lambda\in(0,1)$, we have
\begin{equation} \label{eq:eqeqeq}
\Psi_n^{-1}\big(\gamma_n\big(\lambda K+(1-\lambda)L\big)\big) \geq \sup_{p\in[0,1]} \left(\frac{\lambda}{p}\right)^p \left(\frac{1-\lambda}{1-p}\right)^{1-p}  \Psi_n^{-1}\big(\gamma_n(K)^p \gamma_n(L)^{1-p}\big).
\end{equation}
\end{proposition}

\begin{proof}
Fix $p\in(0,1)$ and let $t_p=\left(\tfrac{\lambda}{p}\right)^p \left(\tfrac{1-\lambda}{1-p}\right)^{1-p} \in [0,1]$. Then, we have the inclusion 
\begin{equation} \label{rem31}
\lambda K + (1-\lambda) L = p\cdot\left(\tfrac{\lambda}{p} K\right) + (1-p) \cdot\left(\tfrac{1-\lambda}{1-p} L\right) \supseteq \left(\tfrac{\lambda}{p} K\right)^p\left(\tfrac{1-\lambda}{1-p}L\right) ^{1-p} = t_p K^pL^{1-p}.
\end{equation}
Since $t_p\in[0,1]$, Lemma \ref{lem:rev-S} implies that
\begin{equation} \label{rem32}
\gamma_n \big(t_p  K^pL^{1-p} \big) \geq \gamma_n\Big(t_p\Psi_n^{-1}\big(\gamma_n \big( K^pL^{1-p}\big) \big) B_2^n\Big)= \Psi_n\Big(t_p\Psi_n^{-1}\big(\gamma_n \big( K^pL^{1-p}\big) \big)\Big),
\end{equation}
Hence combining \eqref{eq:log-bm2} with \eqref{rem31}, \eqref{rem32} and the monotonicity of $\Psi_n^{-1}$, we derive \eqref{eq:eqeqeq}.
\end{proof}

\smallskip

\noindent {\bf 5.} In \cite[p.~5350]{GZ10}, it was shown that inequality \eqref{eq:quest} is not satisfied for $\xi_n=\Psi_n^{-1}$, where $\Psi_n(r)=\gamma_n(rB_2^n)$. As was communicated to us by Ramon van Handel, this observation formally implies that if $M$ is any Borel set in $\R^n$ and $\Xi_n(r) = \gamma_n(rM)$, then \eqref{eq:quest} is not satisfied for $\xi_n=\Xi_n^{-1}$. Indeed, assuming the contrary and plugging $K=aB_2^n$ and $L=bB_2^n$ in \eqref{eq:quest}, we see that $\Xi_n^{-1}\circ \Psi_n$ is concave. Therefore, for every $r\in(0,\infty)$ and $t\in(0,1)$, we have
\begin{equation} \label{rem51}
t \Xi_n^{-1}\big(\Psi_n(r)\big) \leq \Xi_n^{-1}\big(\Psi_n(t r)\big).
\end{equation}
Moreover, Lemma \ref{lem:rev-S} is equivalent to the fact that for every $p\in(0,1)$ and $t\in(0,1)$,
\begin{equation} \label{rem52}
\Xi_n\big(t\Xi_n^{-1}(p)\big) \geq \Psi_n\big(t\Psi_n^{-1}(p)\big)
\end{equation}
Choosing $p=\Psi_n(r)$ in \eqref{rem52} and combining it with \eqref{rem51}, we deduce that $\Xi_n^{-1}\circ\Psi_n$ is an affine function, which readily implies that $\Xi_n(r)=\Psi_n(ar)$ for some $a\in(0,\infty)$, thus contradicting \cite{GZ10}.

This simple argument reveals that the nontrivial equality cases of the sought-for symmetric Ehrhard inequality \eqref{eq:quest} must be more complicated than the one-parameter family $\{sM\}_{s\geq0}$ of dilates of a given Borel set $M$. The possibility of such extremals is reminiscent of the conjectured multiscale solutions of the symmetric Gaussian isoperimetric problem \cite{Bar01, Hei17}.


\bibliography{Gardner-Zvavitch}

\begin{thebibliography}{CEFM04}

\bibitem[Bar01]{Bar01}
F.~Barthe.
\newblock An isoperimetric result for the {G}aussian measure and unconditional
  sets.
\newblock {\em Bull. London Math. Soc.}, 33(4):408--416, 2001.

\bibitem[Bar06]{Bar06}
F.~Barthe.
\newblock The {B}runn-{M}inkowski theorem and related geometric and functional
  inequalities.
\newblock In {\em International {C}ongress of {M}athematicians. {V}ol. {II}},
  pages 1529--1546. Eur. Math. Soc., Z\"{u}rich, 2006.

\bibitem[BK20]{BK20}
K.~J. B\"{o}r\"{o}czky and P.~Kalantzopoulos.
\newblock Log-{B}runn-{M}inkowski inequality under symmetry.
\newblock Preprint available at \url{https://arxiv.org/abs/2002.12239}, 2020.

\bibitem[BL76]{BL76}
H.~J. Brascamp and E.~H. Lieb.
\newblock On extensions of the {B}runn-{M}inkowski and {P}r\'{e}kopa-{L}eindler
  theorems, including inequalities for log concave functions, and with an
  application to the diffusion equation.
\newblock {\em J. Functional Analysis}, 22(4):366--389, 1976.

\bibitem[BLYZ12]{BLYZ12}
K.~J. B\"{o}r\"{o}czky, E.~Lutwak, D.~Yang, and G.~Zhang.
\newblock The log-{B}runn-{M}inkowski inequality.
\newblock {\em Adv. Math.}, 231(3-4):1974--1997, 2012.

\bibitem[Bor03]{Bor03}
C.~Borell.
\newblock The {E}hrhard inequality.
\newblock {\em C. R. Math. Acad. Sci. Paris}, 337(10):663--666, 2003.

\bibitem[CEFM04]{CFM04}
D.~Cordero-Erausquin, M.~Fradelizi, and B.~Maurey.
\newblock The ({B}) conjecture for the {G}aussian measure of dilates of
  symmetric convex sets and related problems.
\newblock {\em J. Funct. Anal.}, 214(2):410--427, 2004.

\bibitem[CHLL18]{CHLL18}
S.~Chen, Y.~Huang, Q.~Li, and J.~Liu.
\newblock ${L}_p$-{B}runn-{M}inkowski inequality for
  $p\in(1-\frac{c}{n^{\frac{3}{2}}},1)$.
\newblock Preprint available at \url{https://arxiv.org/abs/1811.10181}, 2018.

\bibitem[CLM17]{CLM17}
A.~Colesanti, G.~V. Livshyts, and A.~Marsiglietti.
\newblock On the stability of {B}runn-{M}inkowski type inequalities.
\newblock {\em J. Funct. Anal.}, 273(3):1120--1139, 2017.

\bibitem[Col08]{Col08}
A.~Colesanti.
\newblock From the {B}runn-{M}inkowski inequality to a class of
  {P}oincar\'{e}-type inequalities.
\newblock {\em Commun. Contemp. Math.}, 10(5):765--772, 2008.

\bibitem[Dub77]{Dub77}
S.~Dubuc.
\newblock Crit\`eres de convexit\'{e} et in\'{e}galit\'{e}s int\'{e}grales.
\newblock {\em Ann. Inst. Fourier (Grenoble)}, 27(1):x, 135--165, 1977.

\bibitem[Ehr83]{Ehr83}
A.~Ehrhard.
\newblock Sym\'{e}trisation dans l'espace de {G}auss.
\newblock {\em Math. Scand.}, 53(2):281--301, 1983.

\bibitem[Ehr86]{Ehr86}
A.~Ehrhard.
\newblock \'{E}l\'{e}ments extr\'{e}maux pour les in\'{e}galit\'{e}s de
  {B}runn-{M}inkowski gaussiennes.
\newblock {\em Ann. Inst. H. Poincar\'{e} Probab. Statist.}, 22(2):149--168,
  1986.

\bibitem[Gar02]{Gar02}
R.~J. Gardner.
\newblock The {B}runn-{M}inkowski inequality.
\newblock {\em Bull. Amer. Math. Soc. (N.S.)}, 39(3):355--405, 2002.

\bibitem[Gar06]{Gar06}
R.~J. Gardner.
\newblock {\em Geometric tomography}, volume~58 of {\em Encyclopedia of
  Mathematics and its Applications}.
\newblock Cambridge University Press, New York, second edition, 2006.

\bibitem[GT01]{GilbargTrudinger}
D.~Gilbarg and N.~Trudinger.
\newblock {\em Elliptic partial differential equations of second order}.
\newblock Classics in Mathematics. Springer-Verlag Berlin Heidelberg, 2nd
  edition, 2001.

\bibitem[GZ10]{GZ10}
R.~J. Gardner and A.~Zvavitch.
\newblock Gaussian {B}runn-{M}inkowski inequalities.
\newblock {\em Trans. Amer. Math. Soc.}, 362(10):5333--5353, 2010.

\bibitem[Hei17]{Hei17}
S.~Heilman.
\newblock Symmetric convex sets with minimal {G}aussian surface area.
\newblock To appear in {\it Amer. J. Math.}. Preprint available at
  \url{https://arxiv.org/abs/1705.06643}, 2017.

\bibitem[HKL20]{HKL20}
J.~Hosle, A.~Kolesnikov, and G.~Livshyts.
\newblock On the ${L}_p$-{B}runn-{M}inkowski and dimensional
  {B}runn-{M}inkowski conjectures for log-concave measures.
\newblock Preprint available at \url{https://arxiv.org/abs/2003.05282}, 2020.

\bibitem[KL18]{KL18}
A.~Kolesnikov and G.~Livshyts.
\newblock On the {G}ardner-{Z}vavitch conjecture: symmetry in inequalities of
  {B}runn-{M}inkowski type.
\newblock To appear in {\it Adv. Math.}. Preprint available at
  \url{https://arxiv.org/abs/1807.06952}, 2018.

\bibitem[KL20]{KL20}
A.~Kolesnikov and G.~Livshyts.
\newblock On the local version of the log-{B}runn-{M}inkoswki conjecture and
  some new related geometric inequalities.
\newblock Preprint available at \url{https://arxiv.org/abs/2004.06103}, 2020.

\bibitem[KM17]{KM17}
A.~Kolesnikov and E.~Milman.
\newblock Local ${L}^p$-{B}runn-{M}inkowski inequalities for $p<1$.
\newblock To appear in {\it Mem. Amer. Math. Soc.}. Preprint available at
  \url{https://arxiv.org/abs/1711.01089}, 2017.

\bibitem[KM18]{KM18}
A.~V. Kolesnikov and E.~Milman.
\newblock Poincar\'{e} and {B}runn-{M}inkowski inequalities on the boundary of
  weighted {R}iemannian manifolds.
\newblock {\em Amer. J. Math.}, 140(5):1147--1185, 2018.

\bibitem[Lat96]{Lat96}
R.~Lata{\l}a.
\newblock A note on the {E}hrhard inequality.
\newblock {\em Studia Math.}, 118(2):169--174, 1996.

\bibitem[Lat02]{Lat02}
R.~Lata{\l}a.
\newblock On some inequalities for {G}aussian measures.
\newblock In {\em Proceedings of the {I}nternational {C}ongress of
  {M}athematicians, {V}ol. {II} ({B}eijing, 2002)}, pages 813--822. Higher Ed.
  Press, Beijing, 2002.

\bibitem[LMNZ17]{LMNZ17}
G.~Livshyts, A.~Marsiglietti, P.~Nayar, and A.~Zvavitch.
\newblock On the {B}runn-{M}inkowski inequality for general measures with
  applications to new isoperimetric-type inequalities.
\newblock {\em Trans. Amer. Math. Soc.}, 369(12):8725--8742, 2017.

\bibitem[LO99]{LO99}
R.~Lata{\l}a and K.~Oleszkiewicz.
\newblock Gaussian measures of dilatations of convex symmetric sets.
\newblock {\em Ann. Probab.}, 27(4):1922--1938, 1999.

\bibitem[Mar16]{Mar16}
A.~Marsiglietti.
\newblock On the improvement of concavity of convex measures.
\newblock {\em Proc. Amer. Math. Soc.}, 144(2):775--786, 2016.

\bibitem[Mau05]{Mau05}
B.~Maurey.
\newblock In\'{e}galit\'{e} de {B}runn-{M}inkowski-{L}usternik, et autres
  in\'{e}galit\'{e}s g\'{e}om\'{e}triques et fonctionnelles.
\newblock Number 299, pages Exp. No. 928, vii, 95--113. 2005.
\newblock S\'{e}minaire Bourbaki. Vol. 2003/2004.

\bibitem[NT13]{NT13}
P.~Nayar and T.~Tkocz.
\newblock A note on a {B}runn-{M}inkowski inequality for the {G}aussian
  measure.
\newblock {\em Proc. Amer. Math. Soc.}, 141(11):4027--4030, 2013.

\bibitem[Put19]{Put19}
E.~Putterman.
\newblock Equivalence of the local and global versions of the
  ${L}^p$-{B}runn-{M}inkowski inequality.
\newblock Preprint available at \url{https://arxiv.org/abs/1909.03729}, 2019.

\bibitem[Roy14]{Roy14}
T.~Royen.
\newblock A simple proof of the {G}aussian correlation conjecture extended to
  some multivariate gamma distributions.
\newblock {\em Far East J. Theor. Stat.}, 48(2):139--145, 2014.

\bibitem[RYN18]{RY18}
M.~Ritor\'{e} and J.~Yepes~Nicol\'{a}s.
\newblock Brunn-{M}inkowski inequalities in product metric measure spaces.
\newblock {\em Adv. Math.}, 325:824--863, 2018.

\bibitem[Sar16]{Sar16}
C.~Saroglou.
\newblock More on logarithmic sums of convex bodies.
\newblock {\em Mathematika}, 62(3):818--841, 2016.

\bibitem[Sch14]{Sch14}
R.~Schneider.
\newblock {\em Convex bodies: the {B}runn-{M}inkowski theory}, volume 151 of
  {\em Encyclopedia of Mathematics and its Applications}.
\newblock Cambridge University Press, Cambridge, expanded edition, 2014.

\bibitem[YZ19]{YZ19}
Y.~Yang and D.~Zhang.
\newblock The log-{B}runn-{M}inkowski inequality in {$\Bbb{R}^3$}.
\newblock {\em Proc. Amer. Math. Soc.}, 147(10):4465--4475, 2019.

\end{thebibliography}
\bibliographystyle{alpha}

\end{document}